\documentclass[12pt]{article}
\usepackage{fullpage}
\usepackage{url}
\usepackage{amsmath}
\usepackage{comment}
\newcommand{\se}{6e+6}
\newcommand{\qe}{30e+6}
\newcommand{\werthree}{2^{2^{k^{\qe}}}}
\newcommand{\werthreesix}{2^{2^{k^{186}}}}
\newcommand{\werthreeth}{2^{2^{k^{396}}}}
\newcommand{\rakt}{R_a(k,k)}
\newcommand{\ramkmt}{R_{a-1}(k-1,k-1)}
\newcommand{\restritwo}{\Omega((\log\log n )^{1/186})}
\newcommand{\restrithree}{\Omega((\log\log n )^{1/396})}

\newcommand{\M}{m_0}
\newcommand{\Mch}{\binom{m_0}{2}}
\newcommand{\thre}{k_1-3}
\newcommand{\threp}{k_1-2}
\newcommand{\threpp}{k_1-1}
\newcommand{\threppp}{k_1}
\newcommand{\threm}{k_1-4}
\newcommand{\itt}{k_1-3}

\newcommand{\resdone}{\Omega(n^{1/6}(\log n)^{1/3})}
\newcommand{\resd}{\Omega((n^{1/(6d)}(\log n)^{1/3})/d^{1/3})}
\newcommand{\resk}{\frac{(Ck_2)^{6k_1-18}}{(\log k_2)^{2k_1-6}}}
\newcommand{\Rodlns}{R{\"o}dl}
\newcommand{\Erdos}{Erd\H{o}s }
\newcommand{\Erdosns}{Erd\H{o}s}
\newcommand{\degg}{{\rm deg}}

\newcommand{\ang}[1]{\langle#1\rangle}

\newcommand{\nat}{{\sf N}}

\newcommand{\real}{{\sf R}}

\newcommand{\Rpos}{{\sf R}^+}

\newcommand{\xvec}[1]{\ifcase 3{#1} {\ang {x_1,x_2,x_3} } \else 
\ifcase 4{#1} {\ang{x_1,x_2,x_3,x_4}} \else {\ang {x_1,\ldots,x_{#1}}}\fi\fi}
\newcommand{\yvec}[1]{\ifcase 3{#1} {\ang {y_1,y_2,y_3} } \else 
\ifcase 4{#1} {\ang{y_1,y_2,y_3,y_4}} \else {\ang {y_1,\ldots,y_{#1}}}\fi\fi}
\newcommand{\zvec}[1]{\ifcase 3{#1} {\ang {z_1,z_2,z_3} } \else 
\ifcase 4{#1} {\ang{z_1,z_2,z_3,z_4}} \else {\ang {z_1,\ldots,z_{#1}}}\fi\fi}
\newcommand{\vecc}[2]{\ifcase 3{#2} {\ang { {#1}_1,{#1}_2,{#1}_3 } } \else
\ifcase 4{#1} {\ang { {#1}_1,{#1}_2,{#1}_3,{#1}_{4} } }
\else {\ang { {#1}_1,\ldots,{#1}_{#2}}}\fi\fi}
\newcommand{\veccd}[3]{\ifcase 3{#2} {\ang { {#1}_{{#3}1},{#1}_{{#3}2},{#1}_{{#3}3} } } \else
\ifcase 4{#1} {\ang { {#1}_{{#3}1},{#1}_{{#3}2},{#1}_{#3}3},{#1}_{{#3}4} }
\else {\ang { {#1}_{{#3}1},\ldots,{#1}_{{#3}{#2}}}}\fi\fi}
\newcommand{\veccz}[2]{\ifcase 3{#2} {\ang { {#1}_0,{#1}_2,{#1}_3 } } \else
\ifcase 4{#1} {\ang { {#1}_0,{#1}_2,{#1}_3,{#1}_{4} } }
\else {\ang { {#1}_0,\ldots,{#1}_{#2}}}\fi\fi}
\newcommand{\xve}[1]{\ifcase 3{#1} {x_1,x_2,x_3} \else 
\ifcase 4{#1} {x_1,x_2,x_3,x_4} \else {x_1,\ldots,x_{#1}}\fi\fi}
\newcommand{\yve}[1]{\ifcase 3{#1} {y_1,y_2,y_3} \else 
\ifcase 4{#1} {y_1,y_2,y_3,y_4} \else {y_1,\ldots,y_{#1}}\fi\fi}
\newcommand{\zve}[1]{\ifcase 3{#1} {z_1,z_2,z_3} \else 
\ifcase 4{#1} {z_1,z_2,z_3,z_4} \else {z_1,\ldots,z_{#1}}\fi\fi}
\newcommand{\ve}[2]{\ifcase 3#2 {{#1}_1,{#1}_2,{#1}_3} \else
\ifcase 4#2 {{#1}_1,{#1}_2,{#1}_3,{#1}_{4}}
\else {{#1}_1,\ldots,{#1}_{#2}}\fi\fi}
\newcommand{\ved}[3]{\ifcase 3#2 {{#1}_{{#3}1},{#1}_{{#3}2},{#1}_{{#3}3}} \else
\ifcase 4#2 {{#1}_{{#3}1},{#1}_{{#3}2},{#1}_{{#3}3},{#1}_{{#3}4}}
\else {{#1}_{{#3}1},\ldots,{#1}_{{#3}{#2}}}\fi\fi}
\newcommand{\fuve}[3]{
\ifcase 3#2
{{#3}({#1}_1),{#3}({#1}_2,{#3}({#1}_3)} \else
\ifcase 4#2
{{#3}({#1}_1),{#3}({#1}_2),{#3}({#1}_3),{#3}({#1}_4)}
\else
{{#3}({#1}_1),\ldots,{#3}({#1}_{#2})}\fi\fi}
\newcommand{\setmathchar}[1]{\ifmmode#1\else$#1$\fi}
\newcommand{\vlist}[2]{%
	\setmathchar{%
		\compound#2\one{#2}\two
		\ifcompound
			({#1}_1,\ldots,{#1}_{#2})
		\else
			\ifcat N#2
				({#1}_1,\ldots,{#1}_{#2})
			\else
				\ifcase#2
					({#1}_0)\or
					({#1}_1)\or
					({#1}_1,{#1}_2)\or 
					({#1}_1,{#1}_2,{#1}_3)\or
					({#1}_1,{#1}_2,{#1}_3,{#1}_4)\else 
					({#1}_1,\ldots,{#1}_{#2})
				\fi
			\fi
		\fi}}

\newif\ifcompound
\def\compound#1\one#2\two{%
	\def\one{#1}
	\def\two{#2}
	\if\one\two
		\compoundfalse
	\else
		\compoundtrue
	\fi}

\newcommand{\xwe}[1]{\ifcase 3{#1} {x_1\wedge x_2\wedge x_3} \else 
\ifcase 4{#1} {x_1\wedge x_2\wedge x_3\wedge x_4} \else {x_1\wedge \cdots \wedge
x_{#1}}\fi\fi}
\newcommand{\we}[2]{\ifcase 3#2 {\ang { {#1}_1\wedge {#1}_2\wedge {#1}_3 } } \else
\ifcase 4{#1} {\ang { {#1}_1\wedge {#1}_2\wedge {#1}_3\wedge {#1}_{4} } }
\else {\ang { {#1}_1\wedge \cdots\wedge {#1}_{#2}}}\fi\fi}

\newcommand{\st}{\mathrel{:}}

\newcommand{\into}{\rightarrow}

\newcommand{\ep}{\epsilon}
\newcommand{\es}{\emptyset}

\newcommand{\ceil}[1]{\left\lceil {#1}\right\rceil}

\newcommand{\monus}{\;\raise.5ex\hbox{{${\buildrel
    \ldotp\over{\hbox to 6pt{\hrulefill}}}$}}\;}

\newcounter{savenumi}

\newtheorem{theoremfoo}{Theorem}[section] %by chapter in report style
\newenvironment{theorem}{\pagebreak[1]\begin{theoremfoo}}{\end{theoremfoo}}

\newtheorem{lemmafoo}[theoremfoo]{Lemma}
\newenvironment{lemma}{\pagebreak[1]\begin{lemmafoo}}{\end{lemmafoo}}
\newtheorem{conjecturefoo}[theoremfoo]{Conjecture}

\newtheorem{conventionfoo}[theoremfoo]{Convention}

\newtheorem{porismfoo}[theoremfoo]{Porism}

\newtheorem{gamefoo}[theoremfoo]{Game}

\newtheorem{corollaryfoo}[theoremfoo]{Corollary}

\newtheorem{openfoo}[theoremfoo]{Open Problem}

\newtheorem{exercisefoo}{Exercise}

\newcommand{\fig}[1] %usage:\fig{file}
{
 \begin{figure}
 \begin{center}
 \input{#1}
 \end{center}
 \end{figure}
}

\newtheorem{potanafoo}[theoremfoo]{Potential Analogue}

\newtheorem{notefoo}[theoremfoo]{Note}
\newenvironment{note}{\pagebreak[1]\begin{notefoo}\rm}{\end{notefoo}}

\newtheorem{notabenefoo}[theoremfoo]{Nota Bene}

\newtheorem{nttn}[theoremfoo]{Notation}
\newenvironment{notation}{\pagebreak[1]\begin{nttn}\rm}{\end{nttn}}

\newtheorem{empttn}[theoremfoo]{Empirical Note}

\newtheorem{examfoo}[theoremfoo]{Example}

\newtheorem{dfntn}[theoremfoo]{Def}
\newenvironment{definition}{\pagebreak[1]\begin{dfntn}\rm}{\end{dfntn}}

\newtheorem{propositionfoo}[theoremfoo]{Proposition}

\newenvironment{proof}
    {\pagebreak[1]{\narrower\noindent {\bf Proof:\quad\nopagebreak}}}{\QED}

\newcommand{\yyskip}{\penalty-50\vskip 5pt plus 3pt minus 2pt}
\newcommand{\blackslug}{\hbox{\hskip 1pt
        \vrule width 4pt height 8pt depth 1.5pt\hskip 1pt}}
\newcommand{\QED}{{\penalty10000\parindent 0pt\penalty10000
        \hskip 8 pt\nolinebreak\blackslug\hfill\lower 8.5pt\null}
        \par\yyskip\pagebreak[1]}

\newcommand{\BBB}{{\penalty10000\parindent 0pt\penalty10000
        \hskip 8 pt\nolinebreak\hbox{\ }\hfill\lower 8.5pt\null}
        \par\yyskip\pagebreak[1]}

\newtheorem{factfoo}[theoremfoo]{Fact}

\newenvironment{block}{\begin{list}{\hbox{}}{\leftmargin 1em
    \itemindent -1em \topsep 0pt \itemsep 0pt \partopsep 0pt}}{\end{list}}

\begin{document}

\title{Applications of the Canonical Ramsey Theorem to Geometry}

\author{
{William Gasarch}
\thanks{University of Maryland at College Park,
Department of Computer Science,
        College Park, MD\ \ 20742.
\texttt{gasarch@cs.umd.edu}
}
\\ {\small Univ. of MD at College Park}
\and
{Sam Zbarsky}
\thanks{Montgomery Blair High School,
Silver Spring, MD, 20901
\texttt{sa\_zbarsky@yahoo.com}
}
\\ {\small Montgomery Blair High School}
}

\maketitle

\begin{abstract}
Let $\{p_1,\ldots,p_n\}\subseteq \real^d$. We think of $d \le n$.
How big is the largest subset $X$ of points such that all of the distances
determined by elements of $\binom{X}{2}$ are different?
We show that $X$ is at least $\resd$.
This is not the best known; however the technique is new.

Assume that no three of the original points are collinear.
How big is the largest subset $X$ of points such that all of the  areas
determined by elements of $\binom{X}{3}$ are different?
We show that, if $d=2$ then $X$ is at least $\restritwo$,
and if $d=3$ then $X$ is at least $\restrithree$.
We also obtain results for countable sets of points in $R^d$.

All of our results use variants of the canonical Ramsey theorem and
some geometric lemmas. 
\end{abstract}

\section{Introduction}

Let $\{p_1,\ldots,p_n\}\subseteq \real^d$. We think of $d \le n$.
How big is the largest subset $X$ of points such that all of the distances
determined by elements of $\binom{X}{2}$ are different?
Assume that no three of the original points are collinear.
How big is the largest subset $X$ of points such that all of the  areas
determined by elements of $\binom{X}{3}$ are different?

\begin{definition}
Let $a\ge 1$.
Let $h_{a,d}(n)$ be the largest integer so that if $p_1,\ldots,p_n$ are any set
of $n$ distinct points in $\real^d$, no $a$ points in the same $(a-2)$-dimensional space,
then there exists a subset $X$
of $h_{a,d}(n)$ points for which all of the volumes determined by elements of $\binom{X}{a}$
are different.
The {\it $h_{a,d}(n)$ problem} is the problem of establishing upper
and lower bounds on $h_{a,d}(n)$.
The definition extends to letting $n$ be an infinite cardinal $\alpha$ where  $\aleph_0 \le \alpha \le 2^{\aleph_0}$.
\end{definition}

Below we summarize all that is know about $h_{a,d}(n)$ (to our knowledge).

\begin{enumerate}
\item
\Erdosns~\cite{ED-erdos}, in 1946,  showed that the number of distinct differences in the
$\sqrt n \times \sqrt n$ grid is $\le O(\frac{n}{\sqrt{\log n}})$. 
Therefore $h_{2,2}(n) \le O\biggl (\sqrt{\frac{n}{\sqrt{\log n}}}\biggr )$.
For $a\ge 3$ We do not know of any nontrivial upper bounds on $h_{2,d}$.
(The set $n^{1/d} \times \cdots \times n^{1/d}$ has 
many points collinear and hence cannot be
used to obtain an upper bound.)

\item
\Erdosns~\cite{remarks}, in 1950, showed that,
for $\aleph_0\le \alpha \le 2^{\aleph_0}$, $h_{2,d}(\alpha)=\alpha$.

\item
\Erdos considered the $h_{2,d}(n)$ problem 1957~\cite{erdos57} 
and 1970~\cite{erdos70}. 
In the latter paper he notes that $h_{2,2}(7)=3$~\cite{erdoskelly} and
$h_{2,3}(9)=3$~\cite{croft}. \Erdos conjectured that $h_{2,1}(n)=(1+o(n))n^{1/2}$ and notes that
$h_{2,1}(n)\le (1+o(n))n^{1/2}$~\cite{erdosturan}.

\item
Komlos, Sulyok and Szemeredi~\cite{metricone}, in 1975, show that $h_{2,1}(n)\ge \Omega(\sqrt n )$
though they state it in different terms.

\item
\Erdosns~\cite{erdosmetric} considered the $h_{2,d}(n)$ problem in 1986.
He states {\it It is easy to see that $h_{2,d}(n) > n^{\epsilon_d}$ but the best possible value of $\epsilon_d$
is not known. $\epsilon_1=\frac{1}{2}$ follows from a result of Ajtai, Komlos, Sulyok and
Szemeredi~\cite{metricone}.} (We do not know why he added Ajtai who was not an author on that paper.)

\item
Avis, \Erdosns, and Pach~\cite{aep}, in 1991,  showed that for all sets of
$n$ points in the plane, for almost all $k$-subsets $X$ where $k=o(n^{1/7})$,
the elements of $\binom{X}{2}$ determine different distances.
Hence, for example, $h_{2,2}(n) = \Omega(n^{1/7 + \epsilon})$.

\item
Thiele~\cite{points}, 
in his PhD thesis from 1995, has as Theorem 4.33, that for all $d\ge 2$, $h_{2,d}=\Omega(n^{1/(3d-2)})$.

\item
Charalambides~\cite{distinctdist}, in 2012,  showed that $h_{2,2}(n) = \Omega(n^{1/3}/\log n)$.

\item 
We know of no references to $h_{a,d}$ for $a\ge 3$ in the literature.

\item
We believe that this is the first paper to define $h_{a,d}$ in its full generality.

\end{enumerate}

\begin{note}
The problem of $h_{2,2}$ is similar to but distinct from the {\it \Erdos Distance Problem}:
give a set of $n$ points in the plane how many distinct distances are guaranteed.
For more on this problem see~\cite{erdosdistproblem,erdosdistproblemweb}.
The problem of $h_{3,2}$ is similar to but distinct from the problem of determining,
given $n$ points in the plane no three collinear, how many distinct triangle-areas are obtained
(see~\cite{extri} and references therein).
We do not know of any reference to a higher dimensional analog of these 
problems.
\end{note}

Below we list our result. For two of our results stronger results are known and in the above list; however, our proofs are 
very different. We find our proofs simpler.

\begin{itemize}
\item
$h_{2,d}(n) \ge \resd$. (Torsten has a better result.)
\item
$h_{3,2}(n) \ge \restritwo$. 
\item
$h_{3,3}(n) \ge  \restrithree$.
\item
$h_{2,d}(\aleph_0)=\aleph_0$. (\Erdos had a more general result.)
\item
$h_{3,2}(\aleph_0)=\aleph_0$. 
\item
$h_{3,3}(\aleph_0)=\aleph_0$. 
\end{itemize}

Our proofs have two ingredients:
(1) upper bounds on variants of the canonical Ramsey numbers, and
(2) geometric lemmas about points in $\real^d$.

In Section~\ref{se:graphs},\ref{se:lemmas}, and \ref{se:gbounds}
we define terms, prove lemmas, and finally prove 
an upper  bound on a variant of the canonical Ramsey Theorem.
Our proof uses some ideas from the upper bound on the 
standard canonical Ramsey number, $ER(k)$, due to 
Lefmann and \Rodlns~\cite{BetterCanRamsey}.
In Section~\ref{se:geom} we prove a geometric lemma about points in $\real^d$.
In Section~\ref{se:main} 
we use our upper bound and our geometric lemma to prove lower bounds
on $h_{2,d}(n)$.
In Section~\ref{se:tri} we prove the needed variant of the canonical Ramsey theorem,
and 
the needed geometric lemmas, to obtain lower bounds on $h_{3,2}(n)$ and $h_{3,3}(n)$.
In Section~\ref{se:inf} we use known theorems and our geometric lemmas to obtain
results about countable sets of points.
In Section~\ref{se:spec} we speculate about lower bounds for
$h_{a,d}$ for $a\ge 3$.
In Section~\ref{se:open} we list open problems.

\section{Variants of the Canonical Ramsey Theorem}\label{se:graphs}

\begin{notation}
Let $n\in\nat$.
\begin{enumerate}
\item
$[n]$ is the set $\{1,\ldots,n\}$.
\item
If $X$ is a set and $0\le a\le |X|$ then $\binom{X}{a}$ is the set
of all $a$-sized subsets of $X$.
\item
We identify $\binom{X}{a}$ with the complete $a$-ary hypergraph on
the set $X$. Hence we will use terms like {\it vertex} and {\it edge}
when referring to $\binom{X}{a}$.
\item
We will often have $X\subseteq \nat$ and 
a coloring $COL:\binom{X}{a}\into Y$ ($Y$ is either $[c]$ or $\omega$).
We use the usual convention of using $COL(x_1,\ldots,x_a)$ for $COL(\{x_1,\ldots,x_a\})$.
We do not take this to mean that $x_1 < \cdots < x_a$.
\end{enumerate}
\end{notation}

We define terms and then state the canonical Ramsey theorem (for graphs).
It was first proven by \Erdos and Rado~\cite{Canramsey}.
The best known upper  bounds on the canonical Ramsey numbers are due to
Lefmann and \Rodlns~\cite{BetterCanRamsey}.

\begin{definition}
Let $COL:\binom{[n]}{2}\into\omega$.
Let $V\subseteq [n]$.
\begin{enumerate}
\item
The set $V$ is {\it homogenous} (henceforth {\it homog}) if 
for all $x_1<x_2$ and $y_1<y_2$
$$COL(x_1,x_2)=COL(y_1,y_2) \hbox{ iff } TRUE.$$
(Every edge in $\binom{V}{2}$ is colored the same.)
\item
The set $V$ is {\it min-homogenous} 
(henceforth {\it min-homog}) 
if for all $x_1<x_2$ and $y_1<y_2$
$$COL(x_1,x_2)=COL(y_1,y_2) \hbox{ iff } x_1=y_1.$$
\item
The set $V$ is {\it max-homogenous} 
(henceforth {\it max-homog}) 
if for all $x_1<x_2$ and $y_1<y_2$
$$COL(x_1,x_2)=COL(y_1,y_2) \hbox{ iff } x_2=y_2.$$
\item
The set $V$ is {\it rainbow} 
if for all $x_1<x_2$ and $y_1<y_2$
$$COL(x_1,x_2)=COL(y_1,y_2) \hbox{ iff } (x_1=y_1 \hbox{ and }x_2=y_2).$$
(Every edge in $\binom{V}{2}$ is colored differently.)
\end{enumerate}
\end{definition}

\begin{theorem}
For all $k$ there exists $n$ such that, 
for all colorings of $\binom{[n]}{2}$
there is either a homog set of size $k$,
a min-homog set of size $k$, 
a max-homog set of size $k$, 
or a rainbow set of size $k$.
We denote the least value of $n$ that works by $ER(k)$.
\end{theorem}

We now state the asymmetric canonical Ramsey Theorem.

\begin{theorem}\label{th:asyer}
For all $k_1,k_2$ there exists $n$ such that, 
for all colorings of $\binom{[n]}{2}$,
there is either a homog set of size $k_1$,
a min-homog set of size $k_1$, 
a max-homog set of size $k_1$, 
or a rainbow set of size $k_2$.
We denote the least value of $n$ that works by $ER(k_1,k_2)$.
\end{theorem}

We will actually need a variant of the asymmetric canonical Ramsey Theorem which is weaker
but gives better upper bounds. 

\begin{definition}\label{de:weak}
Let $COL:\binom{[n]}{2}\into\omega$.
Let $V\subseteq [n]$.
The set $V$ is {\it weakly homogenous} (henceforth {\it whomog}) if 
there is a way to linear order $V$ (not necessarily the numerical order),
$$V= \{ x_1, x_2, \ldots, x_L \},$$
such that, for all for all $1\le i\le L-3$, for all $i<j<k\le L$, 
$$COL(x_i,x_j)=COL(x_i,x_k).$$
Informally, the color of $(x_i,x_j)$, where $i<j$, depends only on $i$.
(We intentionally have $1\le i\le L-3$. We do not care if 
$COL(x_{L-2},x_{L-1})=COL(x_{L-2},x_L)$.)
\end{definition}

\begin{note}
When presenting a whomog set we will also present the needed
linear order.
\end{note}

The following theorem follows from~\ref{th:asyer}.

\begin{theorem}\label{th:introer}
For all $k_1,k_2$ there exists $n$ such that, 
for all colorings of $\binom{[n]}{2}$
there is either 
a whomog set of size $k_1$, 
or a rainbow set of size $k_2$.
We denote the least value of $n$ that works by $WER(k_1,k_2)$.
\end{theorem}

In Theorem~\ref{th:gbounds} we will show
$$WER(k_1,k_2) \le \resk.$$

\section{Lemma to Help Obtain Rainbow Sets}\label{se:lemmas}

The next definition and lemmas gives a way to get a rainbow set under some conditions.

\begin{definition}
Let $COL:\binom{[m]}{2}\into\omega$.
If $c$ is a color and $v\in [m]$
then $\degg_{c}(v)$ is the number of $c$-colored edges
with an endpoint in $v$.
\end{definition}

\begin{comment}

The following theorem is due to Babai~\cite{maximal}.
We include the proof since the paper is not available on-line.

\begin{lemma}\label{le:maximal}~
Let $m\ge 3$.
Let $COL:\binom{[m]}{2}\into\omega$ be such that,
for all $v\in [m]$ and all colors $c$,
$\degg_c(v)\le 1$. Then 
there exists a rainbow set of size 
$\ge (2m)^{1/3}$.
\end{lemma}

\begin{proof}

Let $X$ be a maximal rainbow set. This means that, 

$$(\forall y\in [m]- X)[X\cup \{y\} \hbox{ is not a rainbow set}].$$

Let $y\in [m]-X$. Why is $y\notin X$?
One of the following must occur:
\begin{enumerate}
\item
There exists $u,u_1,u_2\in X$ such that  $u_1\ne u_2$ and 
$COL(y,u)=COL(u_1,u_2)$.
\item
There exists $u_1\ne u_2\in X$ such that 
$COL(y,u_1)=COL(y,u_2)$.
This cannot happen since then $y$ has some color degree $\ge 2$.
\end{enumerate}

We map $[m]-X$ to $X\times \binom{X}{2}$ by mapping $y\in [m]-X$ to 
$(u,\{u_1,u_2\})$ as indicated in item 1 above.
This map is injective since if $y_1$ and $y_2$ both map to $(u, \{u_1,u_2\})$
then $COL(y_1,u)=COL(y_2,u)$.

This map has domain of size $m-|X|$ and co-domain of size $|X|\binom{|X|}{2}$.
Hence

$$m-|X| \le |X|\binom{|X|}{2} = |X|^2(|X|-1)/2 = \frac{|X|^3-|X|^2}{2} \le \frac{|X|^3}{2} - |X|$$

$$m \le \frac{|X|^3}{2}.$$

$$|X|\ge (2m)^{1/3}.$$

\end{proof}

\end{comment}

The following result is due to Alon, Lefmann, and \Rodlns~\cite{maximalbest}.

\begin{lemma}
Let $m\ge 3$.
\begin{enumerate}
\item
Let $COL:\binom{[m]}{2}\into\omega$ be such that,
for all $v\in [m]$ and all colors $c$,
$\degg_c(v)\le 1$. Then
there exists a rainbow set of size 
$\ge \Omega((m\log m)^{1/3})$.
\item
There exists a coloring of $\binom{[m]}{2}$
such that for all $v\in [m]$ and all colors $c$,
$\degg_c(v)\le 1$ and all rainbow sets are of size
$\le O((m\log m)^{1/3})$.
\end{enumerate}
\end{lemma}

The following easily follows:

\begin{lemma}\label{le:maximal}
Let $m\ge 3$.
Let $COL:\binom{[m]}{2}\into\omega$ be such that,
for all $v\in [m]$ and all colors $c$,
$\degg_c(v)\le 1$. 
If $m=\Omega ( \frac{k^3}{\log k} )$ 
then there exists a rainbow set of size $k$.
\end{lemma}

The following definitions and lemmas will be  used to achieve the premise of Lemma~\ref{le:maximal}

\begin{definition}
Let $COL:\binom{[m]}{2}\into\omega$.
Let $c$ be a color and let $x\in[m]$.
\begin{enumerate}
\item
$\degg_{c}(x)$ is the number of $c$-colored edges $(x,y)$.
\item
A {\it bad triple} is a triple $a,b,c$ such that ${a,b,c}$ does not form a rainbow $K_3$.
\end{enumerate}
\end{definition}

The next two lemmas show us how to, in some cases,
reduce the number of bad triples.

\begin{lemma}\label{le:bad}
Let $COL:\binom{[m]}{2}\into\omega$ be such that, for every 
color $c$ and vertex $v$, $\degg_{c}(v)\le d$.
Then the number of bad triples is less than $\frac{dm^2}{6}$.
\end{lemma}

\begin{proof}
We assume that $d$ divides $m-1$. We leave the minor adjustment needed in case $d$ does
not divide $m-1$ to the reader.

Let $b$ be the number of bad triples. We upper bound $b$ by summing over all $v$ that are
the point of the triple with two same-colored edges coming out of it.
This actually counts each triple thrice. Hence we have

\[
\begin{array}{rl}
3b & \le\sum_{v\in [m]}\sum_{c\in\nat}\hbox{ Num of bad triples $\{v,u_1,u_2\}$ with $COL(v,u_1)=COL(v,u_2)=c$ } \cr
  & \le\sum_{v\in [m]}\sum_{c\in\nat}\binom{\degg_c(v)}{2}  \cr
\end{array}
\]

We bound the inner summation.  Since $v$ is of degree $m-1$ we can 
renumber the colors as $1,2,\ldots,m-1$.
Since $\sum_{c=1}^{m-1} \degg_c(v) = m-1$ and $(\forall c)[\degg_c(v)\le d]$ the sum
$\sum_{c=1}^{m-1}\binom{\degg_c(v)}{2}$
is maximized when $d=\degg_1(v)=\degg_2(v)=\cdots=\degg_{(m-1)/d}(v)$ and the rest of the $\degg_c(v)$'s are 0.
Hence 

$$
3b\le\sum_{v\in [m]}\sum_{c=1}^{m-1} \binom{\degg_c(v)}{2} 
\le \sum_{v\in[m]}\sum_{c=1}^{(m-1)/d}  \binom{d}{2} < \frac{dm^2}{2}.
$$

Hence $b\le \frac{dm^2}{6}$.
\end{proof}

\begin{lemma}\label{le:lessbad}
Let $COL:\binom{[m]}{2}\into\omega$ be such that there are $\le b$ bad triples.
Let $1\le m'\le m$.
There exists an $m'$-sized set of vertices with $\le b\bigl (\frac{m'}{m}\bigr )^3$ bad triples.
\end{lemma}

\begin{proof}
Pick a set $X$ of size $m'$ at random. 
Let $E$ be the expected number of bad triples.
Note that

$$E = \sum_{\{v_1,v_2,v_3\} \hbox{ bad } } \hbox{Prob that $\{v_1,v_2,v_3\} \subseteq X $ }. $$

Let $\{v_1,v_2,v_3\}$ be a bad triple. 
The probability that all three nodes are in $X$ is bounded by

$$
\frac{\binom{m-3}{m'-3}}{\binom{m}{m'}}\le
\frac{m'(m'-1)(m'-2)}{m(m-1)(m-2)}\le\biggl(\frac{m'}{m}\biggr )^3.
$$ 

Hence the expected number of bad triples is $\le b(\frac{m'}{m})^3$.
Therefore there must exist some $X$ that has
$\le b(\frac{m'}{m})^3$ bad triples.

\end{proof}

\begin{note}
The above theorem presents the user with an interesting tradeoff.
She wants a large set with few bad triples. If $m'$ is large then
you get a large set, but it will have many bad triples.
If $m'$ is small then you won't have many bad triples, but
$m'$ is small. We will need a Goldilocks-$m'$ that is just right.
\end{note}

\section{The Asymmetric Weak Canonical Ramsey Theorem}\label{se:gbounds}

\begin{theorem}\label{th:gbounds}
There exists $C$ such that, for all $k_1,k_2$,
$$WER(k_1,k_2) \le \resk.$$
\end{theorem}

\begin{proof}

Let $n,m,m',m'',\delta$ be parameters to be determined later. They will be functions
of $k_1,k_2$.
Let $COL:\binom{[n]}{2}\into\omega$.

\noindent
{\bf Intuition:} In the usual proofs of Ramsey's Theorem (for two colors) we take a vertex $v$ and
see which of
$\degg_{RED}(v)$ or $\deg_{BLUE}(v)$ is large. One of them
must be at least half of the size of the vertices still in play.
Here we change this up:
\begin{itemize}
\item
Instead of taking a particular vertex $v$ we ask if there is {\it any}
$v$ and {\it any} color $c$ such that $\degg_c(v)$ is large.
\item
What is large? Similar to the proof of Ramsey's theorem it will be a fraction of
what is left.
Unlike the proof of Ramsey's theorem this fraction, $\delta$, will depend on $k_2$.
\item
In the proof of Ramsey's theorem we were guaranteed that one of $\degg_{RED}(v)$
or $\degg_{BLUE}(v)$ is large. Here we have no such guarantee.
We may fail. In that case something else happens and leads to a rainbow set!
\end{itemize}

\noindent
CONSTRUCTION

\noindent
{\bf Phase 1:} 

\noindent
{\bf Stage 0:}
\begin{enumerate}
\item
$V_0=\es$.
\item
$N_0=[n]$. 
\item
$COL'$ is not defined on any points.
\end{enumerate}

\noindent
{\bf Stage i:}
Assume that $V_{i-1}=\{x_1,\ldots,x_{i-1}\}$, $c_1,\ldots,c_{i-1}$,  and $N_{i-1}$ are already defined.

If there exists $x\in N_{i-1}$ and $c$ a color such that $\degg_c(x)\ge \delta N_{i-1}$
then do the following:
\[
\begin{array}{rl}
V_i = & V_{i-1} \cup \{x\}\cr
N_i   = & \{ v \in N_{i-1} \st COL(x,v)=c \} \cr
x_i  = & x \cr
c_i=  & c \cr
\end{array}
\]
Note that $|N_i| \ge \delta |N_{i-1}|$,
so $|N_i| \ge \delta^i n$,
and $|V_i|=i$.
If $i=\thre$ then goto Phase $2$.

If no such $x,c$ exist then goto Phase $3$.
In this case we formally regard the jump to Phase 3 as
happening in stage $i-1$ since nothing has changed.

\noindent
{\bf End of Phase 1}

\bigskip

\noindent
{\bf Phase 2:}
Since we are in Phase 2 $i=\thre$.
Let
$$V=V_{k_1-3}=\{ x_1, x_2, \ldots, x_{\thre} \}.$$
(This is the order the elements came into $V,$ not the numeric order.)
By construction $V$ is a whomog set of size $\thre$.
Note that, for all elements $x\in N_{\thre}$, $COL(x_i,x)=c_i$.

We need $|N_{\thre}|\ge 3$ (you will see why soon).
Since $|N_{\thre}|\ge \delta^{\thre}n$ we satisfy 
$|N_{\thre}|\ge 3$
by imposing the 
constraint
$$n\ge \frac{3}{\delta^{\thre}}.$$

Let 
$x_{\threp}$, $x_{\threpp}$, and $x_{\threppp}$ be three points from $N_{\thre}$.
Let $H$ be (in this order)

$$H= \{ x_1, x_2, \ldots, x_{\thre}, x_{\threp}, x_{\threpp}, x_{\threppp} \}.$$
$H$ is clearly whomog.
(Recall that in a whomog set of size $\threppp$ 
we do not care if $COL(x_{\threp},x_{\threpp})=COL(x_{\threp},x_{\threppp})$.)

\noindent
{\bf End of Phase 2}

\bigskip

\noindent
{\bf Phase $3$:}
Since we are in Phase 2 $i\le \threm$.
Let $N=N_i$.
$$|N| \ge \delta^i n \ge \delta^{\threm} n.$$
We will need $|N|\ge m$ since  we will find a rainbow subset of
$N$ and need $N$ to be big in the first place
so that the rainbow subset is of size at least $k_2$.
Hence we impose the constraint

$$n\ge \frac{m}{\delta^{\threm}}.$$

Recall that we also imposed the constraint $n\ge \frac{3}{\delta^{\thre}}$.
To satisfy both of these constraints we impose the following two
constraints:
$$m=\frac{3}{\delta}$$
and
$$n=\frac{3}{\delta^{\thre}}.$$

Let $|N|=\M$. 
We have no control over $\M$. All we will know is that $m\le \M\le n$.
Later on $\M$ will cancel out of calculations and hence we can set
other parameters independent of it.

Let $COL$ be the coloring restricted to $\binom{N}{2}$.
We can assume the colors are a subset of $\{1,\ldots,\Mch\}$.
Since we are in Phase 3 we know that,
for all $v\in N$, for all colors $c$,
$\degg_c(v) \le \delta \M$.
\begin{comment}

Note also that, for any vertex $v\in N$, 

$$\M-1=\sum_{c=1}^{\Mch}\degg_c(v) \le \sum_{c=1}^{\Mch} \delta m \le \delta m_0^2m.$$

Hence we need

$$\delta > \frac{\M-1}{\M^2m}.$$

Since $\M\ge m$ the constraint $m=\frac{3}{\delta}$ already implies this.

Note that $COL$ is a coloring of $\binom{N}{2}$ such that
for every $v\in N$ and color $c$, $\degg_c(v)\le \delta \M$. 
\end{comment}
Hence, by Lemma~\ref{le:bad}, there are at most 

$$\frac{\delta \M \times \M^2}{6}\le \delta \M^3$$

\noindent
bad triples (we ignore the denominator of 6 since it makes later
calculations easier and only affects the constant).

By Lemma~\ref{le:lessbad} there exists $X\subseteq N$ of size $m'$ that has
$$b<\delta \M^3 \times \biggl (\frac{m'}{\M}\biggr)^3 = \delta (m')^3$$
bad triples. Note that the bound on $b$ is independent of $\M$.

We set $m'$ such that the number of bad triples is so small that
we can just remove one point from each to obtain a set  $X$  of size $m'$  with
{\it no} bad triples.

Since the number of bad triples is $\le \delta(m')^3$ we need

$$m' - \delta(m')^3 \ge m''.$$

Hence we impose the constraint

$$\delta =\frac{m'-m''}{(m')^3}.$$

We will now set the parameters.
Since we will use Lemma~\ref{le:maximal} it would be difficult to
optimize the parameters.
Hence we pick parameters that are easy to work with.

We will use Lemma~\ref{le:maximal} on $X$ to obtain a rainbow set of size $k_2$.
Hence we take

$$m''= \frac{Ak_2^3}{\log k_2}$$

\noindent
where $A$ is chosen to (1) make $m''$ large 
enough to satisfy the premise of Lemma~\ref{le:maximal}, 
(2) make $m''$ an integer, and
(3) make $m',m,n$, which will be functions of $m''$, integers.

We take 

$$m'=1.5m''. \hbox{ This is the value that minimize $\delta$ though this does not matter.}$$
With this value of $m'$ we obtain

$$
\delta = \frac{m'-m''}{(m'')^3} = \frac{1}{B(m'')^2}
$$

\noindent
where $B$ is an appropriate constant. Our constraints force

\[
\begin{array}{rl}
m=&  \frac{3}{\delta}\cr
n=&  \frac{3}{\delta^{\itt}} = 3(Bm'')^{2(\itt)} = \resk\cr
\end{array}
\]

\noindent
where $C$ is an appropriate constant.
\end{proof}

\section{Lemmas from Geometry}\label{se:geom}

\begin{definition}\label{de:cool}
Let $d\in\nat$.
\begin{enumerate}
\item
If $p,q\in \real^d$ then let $|p-q|$ be the Euclidean distance
between $p$ and $q$.
\item
Let $p_1,\ldots,p_{n}$ be points in $\real^d$.
$(p_1,\ldots,p_n)$ is a {\it cool sequence} if,
for all $1\le i \le n-3$, for all $i< j \le n$, $|p_i-p_j|$
is determined solely by $p_i$.
(Formally: for all $1\le i\le n-3$ there exists $L_i$ such that,
for all $i+1\le j\le n, |p_i - p_j|=L_i$.)
We intentionally have $1\le i\le n-3$.
We do not care if $|p_{n-2}-p_{n-1}|=|p_{n-2}-p_n|$.
\item
The sphere with center $x\in \real^{d+1}$ and radius $r\in \Rpos$ is the set
$$ \{ y\in \real^{d+1} \st |x-y|=r \}.$$
If the sphere is completely contained in an $(n+1)$-dimensional plane then the sphere is
called an $n$-sphere. 
\end{enumerate}
\end{definition}

Note that if $(p_1,\ldots,p_n)$ is cool then
$(p_2,\ldots,p_n)$ is cool. We use this implicitly without mention.

The following lemma is well known. 

\begin{lemma}\label{le:sphere}
Let $S$ be a $d$-sphere. Let $x\in S$ and $r\in\Rpos$.
The set
$$\{ y\in S \st |x-y|=r \}$$
is either an $(d-1)$-sphere or is empty.
\end{lemma}

\begin{lemma}\label{le:full}
For all $d\ge 0$
there does not exist a cool sequence $p_1,\ldots,p_{d+3}$ on a $d$-sphere.
\end{lemma}

\begin{proof}
We prove this by induction on $d$.

\noindent
{\bf Base Case $d=0$:}
Assume, by way of contradiction, that $(p_1,p_2,p_3)$ 
form a cool sequence on a 0-sphere.
A 0-sphere is a set of two points, hence this is impossible.
(Note that being a cool sequence did not constraint $(p_1,p_2,p_3)$ at all.)

\noindent
{\bf Induction Hypothesis:} The theorem holds for $d-1$.

\noindent
{\bf Induction Step:} We prove the theorem for $d$. We may assume $d\ge 1$.
Assume, by way of contradiction, that $(p_1,\ldots,p_{d+2})$ form a cool sequence 
on an $d$-sphere.
Since $|p_1-p_2|=|p_1-p_3|=\cdots=|p_1-p_{d+2}|$ we know, by Lemma~\ref{le:sphere}, that
$p_2,p_3,\ldots,p_{d+2}$ are on an $(d-1)$-sphere.
Since $p_2,\ldots,p_{d+2}$ is a cool sequence 
this is impossible by the induction hypothesis.
\end{proof}

\begin{note}
The following related statement is well known:
{\it if there are $d+2$ points in $\real^d$ then
it is not the case that all $\binom{d+2}{2}$ distances
are the same.}
We have not been able to locate this result in an old fashion
journal (perhaps its behind a paywall); however, there is a proof
at mathoverflow.net here:

\noindent
\url{http://mathoverflow.net/questions/30270/}

\noindent
\url{maximum-number-of-mutually-equidistant-points-}

\url{in-an-n-dimensional-euclidean-space}

\end{note}

\begin{lemma}\label{le:nohomog}
Let $d\in\nat$.
Let $p_1,\ldots,p_{n}$ be points in $\real^d$.
Color $\binom{[n]}{2}$ via
$COL(i,j)=|p_i-p_j|$.
This coloring has no whomog set of size $d+3$.
\end{lemma}

\begin{proof}
Assume, by way of contradiction, that there exists a whomog set of size $d+3$.
By renumbering we can assume the whomog set is $[d+3]$.
Clearly $p_1,\ldots,p_{d+3}$ form a cool sequence.
Note that our not-caring about $COL(d+1,d+2)=COL(d+1,d+3)$ in the definition of whomog
is reflected in our not-caring about $|p_{d+1}-p_{d+2}|=|p_{d+1} - p_{d+3}|$ in the definition of a cool
sequence.

Since $|p_1-p_2|=|p_1-p_3|=\cdots=|p_1-p_{d+3}|$,
$p_2,\ldots,p_{d+3}$ are on the $(d-1)$-sphere (centered at $p_1$).
This contradicts Lemma~\ref{le:full}.
\end{proof}

\section{Lower Bound on $h_{2,d}(n)$}\label{se:main}

We defined $WER(k_1,k_2)$ in terms of colorings with co-domain $\omega$.
In our application we will actually use colorings with co-domain $\Rpos$.
The change in our results to accommodate this is only
a change of notation.  Hence we use our lower bounds on $WER(k_1,k_2)$
in this context without mention.

\begin{theorem}
For all $d\ge 1$, $h_{2,d}(n)=\resd$.
\end{theorem}

\begin{proof}
Let $P=\{p_1,\ldots,p_n\}$ be $n$ points in $\real^d$.
Let $COL:\binom{[n]}{2}\into\real$
defined by  $COL(i,j)=|p_i-p_j|$.

Let $k$ be the largest integer such that
$n\ge WER(d+3,k)$.
By Theorem~\ref{th:gbounds} it will suffice to take $k=\resd$.
By the definition of $WER_3(d+3,k)$
there is either a whomog set of size $d+3$ or a rainbow
set of size $k$. By Lemma~\ref{le:nohomog} there cannot be such a whomog set,
hence must be a rainbow set of size $k$.
\end{proof}

\section{Lower Bounds on $h_{3,2}$ and $h_{3,3}$}\label{se:tri}

For the problem of $h_{2,d}$ we used (1) upper bounds on the asymmetric weak canonical Ramsey theorem and
(2) a geometric lemma.
Here we will use the same approach though our version of the asymmetric weak canonical Ramsey theorem 
does not involve reordering the vertices.

\subsection{The Asymmetric 3-ary Canonical Ramsey Theorem}

\begin{definition}
Let $COL:\binom{[n]}{a}\into\omega$.
Let $V\subseteq [n]$.
\begin{enumerate}
\item
Let  $I\subseteq [a]$.
The set $V$ is {\it $I$-homogenous} (henceforth {\it $I$-homog}) if 
for all $x_1<\cdots<x_a\in\binom{[n]}{a}$ and $y_1<\cdots<y_a \in \binom{[n]}{a}$,
$$(\forall i\in I)[x_i=y_i] \hbox{ iff } COL(x_1,\ldots,x_a)=COL(y_1,\ldots,y_a).$$
Informally, the color of an element of $\binom{[n]}{a}$ depends exactly on the coordinates in $I$.

\item
The set $V$ is {\it rainbow} if 
every edge in $\binom{V}{a}$ is colored differently.
Note that this is just an $I$-homog set where $I=[a]$.
\end{enumerate}
\end{definition}

We will need the asymmetric hypergraph Ramsey numbers and $a$-ary \Erdosns-Rado numbers.

\begin{definition}
Let $a\ge 1$.
Let $k_1,k_2,\ldots,k_c\ge 1$. 
\begin{enumerate}
\item
Let $COL:\binom{[n]}{a}\into [c]$.
(Note that there is a bound on the number of colors.)
Let $V\subseteq [n]$. The set $V$ is {\it homog with color $i$}
if $COL$ restricted to $\binom{V}{a}$ always returns $i$.
\item
$R_a(k_1,k_2,\ldots,k_c)$ is
the least $n$ such that, for all $COL:\binom{[n]}{a}\into [c]$,
there exists $1\le i\le c$ and a homog set of size $k_i$ with color $i$.
$R_a(k_1,k_2,\ldots,k_c)$ is known to exist by the hypergraph Ramsey theorem.
\item
$ER_a(k_1,k_2)$ is
the least $n$ such that, for all $COL:\binom{[n]}{a}\into \omega$,
there exists either (1) an $I\subset [a]$ (note that this is a proper subset) 
and an $I$-homog set of size $k_1$, or
(2) a rainbow set of size $k_2$.
$ER_a(k_1,k_2)$ is known to exist by the $a$-ary canonical Ramsey theorem.
\end{enumerate}
\end{definition}

\begin{definition}\label{de:weak3}
Let $a\ge 3$.
\begin{enumerate}
\item
Let $COL:\binom{[n]}{a}\into\omega$.
Let $V\subseteq [n]$. Let $I\subset [a]$.
The set $V$ is {\it $I$-weakly homogenous} (henceforth {\it $I$-whomog}) if 
for all $x_1,\ldots,x_a,y_1,\ldots,y_a\in[n]$
$$(\forall i\in I)[x_i=y_i] \implies COL(x_1,\ldots,x_a)=COL(y_1,\ldots,y_a).$$
(Note that this differs slightly from the $a=2$ case it that we
do not change around the ordering.)
\item
Let $COL:\binom{[n]}{a}\into\omega$.
Let $V\subseteq [n]$. 
The set $V$ is {\it weakly homogenous} (henceforth {\it whomog}) if 
there is an $I\subset [a]$ (note that this is proper subset) such that
$V$ is $I$-whomog.
\item
Let $k_1,k_2\in\nat$. We denote the least $n$ such that,
for all $COL:\binom{[n]}{a}\into\omega$, there is either a whomog set of size $k_1$
or a rainbow set of size $k_2$, by $WER_a(k_1,k_2)$.
$WER_a(k_1,k_2)$ is known to exist by the $a$-ary canonical Ramsey theorem.
\end{enumerate}
\end{definition}

\begin{note}
Note that if a set is $\{1\}$-whomog then its also $\{1,2\}$-whomog.
\end{note}

A modification of the bound on $ER_3(k)$ by Lefmann and \Rodlns~\cite{canRamsey3} yields

$$
ER_3(k_1,k_2) \le R_4(6,6,6,6,k_1,k_1,k_1,k_1,\ceil{\frac{k_1^3}{4}},\ceil{\frac{k_1^3}{4}},2k_1^3,
\ceil{\frac{k_2^5}{36}}).
$$

We get better bounds on $WER_3(k_1,k_2)$.

We first need the $k=3$ case of Lemma 3 of~\cite{canRamsey3}
which we state:

\begin{lemma}\label{le:cap}
Let $COL:\binom{X}{3}\into\omega$ be 
such that, for all $S,T\in \binom{X}{3}$, with $|S\cap T|=2$,
$COL(S)\ne COL(T)$. Then there exists a rainbow set of size
$\ge \Omega(|X|^{1/5})$ .
\end{lemma}

\begin{lemma}\label{le:wer3}
$WER_3(k_1,k_2)\le R_4(k_1,k_1+2,k_1+2,k_1,k_1+2,k_1,k_2^5)$
\end{lemma}

\begin{proof}
Let $n=R_4(k_1,k_1+2,k_1+2,k_1,k_1+2,k_1,k_2^5)$.

We are given $COL:\binom{[n]}{3}\into \omega$.
We use $COL$ to obtain a $COL':\binom{[n]}{4}\into [7]$.
We will use the (ordinary) 3-ary Ramsey theorem.

We define $COL'(x_1<x_2<x_3<x_4)$ by looking at $COL$ on all
$\binom{4}{3}$ triples of $\{x_1,x_2,x_3,x_4\}$ and
see how their colors compare to each other.

For each case we assume the negation of all the prior cases.
In each case, we indicate what happens if this is the color of the
infinite homog set.

In all the cases below we use the following notation: 
if we are referring to a set $X$ 
and $x\in X$ then $x^+$ is the next element of $X$ after $x$.

\begin{enumerate}

\item
If $COL(x_1,x_2,x_3)=COL(x_1,x_2,x_4)$ then $COL'(x_1,x_2,x_3,x_4)=1$.
Assume $X$ is a homog set of size $k_1$ with color 1.
Clearly $X$ is $\{1,2\}$-whomog for $COL$.

\item
If $COL(x_1,x_2,x_3)=COL(x_1,x_3,x_4)$ then $COL'(x_1,x_2,x_3,x_4)=2$.
Assume $X$ is a homog set of size $k_1+2$ with  color 2.
Let $z_1,z_2$ be the largest two elements of $X$.
We show that $X-\{z_1,z_2\}$ is $\{1\}$-whomog for $COL$.
Assume (1) $x_1<x_2<x_3$, (2) $x_1<y_2<y_3$, and
(3) $x_1,x_2,x_3,y_2,y_3\in X-\{z_1,z_2\}$.
We need $COL(x_1,x_2,x_3)=COL(x_1,y_2,y_3)$. 
$$COL(x_1,x_2,x_3)=COL(x_1,x_3,x_3^+)=\cdots=COL(x_1,z_1,z_2)$$
and
$$COL(x_1,y_2,y_3)=COL(x_1,y_3,y_3^+)=\cdots=COL(x_1,z_1,z_2).$$
Hence they equal each other. 

\item
If $COL(x_1,x_2,x_3)=COL(x_2,x_3,x_4)$ then $COL'(x_1,x_2,x_3,x_4)=3$.
Assume $X$ is a homog set of size $k_1+2$ with  color 3.
Let $z_1,z_2$ be the largest two elements of $X$.
We show that $X-\{z_1,z_2\}$ is $\es$-whomog for $COL$ 
(all edges are the same color).
Note that all triples of the form $(x,x^+,x^{++}\}$ have the same color.  Denote that color $RED$.
Assume (1) $x_1<x_2<x_3$, (2) $y_1<y_2<y_3$, 
and (3) $x_1,x_2,x_3,y_1,y_2,y_3\in X-\{z_1,z_2\}$.
We need $COL(x_1,x_2,x_3)=COL(y_1,y_2,y_3)$. 
Note that
$$COL(x_1,x_2,x_3)=COL(x_2,x_3,x_3^+)=COL(x_3,x_3^+,x_3^{++})=RED.$$
By the same reasoning $COL(y_1,y_2,y_3)=RED$.
(Note that it is possible that $x_3^+,x_3^{++},y_3^+,y_3^{++}\in \{z_1\}$ or
$x_3^{++},y_3^{++}\in \{z_2\}$.)

\item
If $COL(x_1,x_2,x_4)=COL(x_1,x_3,x_4)$ then $COL'(x_1,x_2,x_3,x_4)=4$.
Assume $X$ is a homog set of size $k_1$ with color 4.
Clearly $X$ is $\{1,3\}$-whomog for $COL$.

\item
If $COL(x_1,x_2,x_4)=COL(x_2,x_3,x_4)$ then $COL'(x_1,x_2,x_3,x_4)=5$.
Assume $X$ is a homog set of size $k_1+2$ with color 5.
Let $z_1,z_2$ be the smallest elements of $X$.
Then $X-\{z_1,z_2\}$ is $\{3\}$-whomog for $COL$ by the same reasoning as in part 2.

\item
If $COL(x_1,x_3,x_4)=COL(x_2,x_3,x_4)$ then $COL'(x_1,x_2,x_3,x_4)=6$.
Assume $X$ is a homog set of size $k_1$ with color 6.
Clearly $X$ is $\{2,3\}$-whomog for $COL$.

\item
If none of the above happen then $COL'(x_1,x_2,x_3,x_4)=7$.
Assume $X$ is a homog set of size $k_2^5$ with color 7.
$COL$ restricted to $\binom{X}{3}$ 
satisfies the premise of Lemma~\ref{le:cap}:
If $S,T\in \binom{X}{3}$ with $|S\cap T|=2$ then since
$COL'(S\cup T)\notin \{1,2,3,4,5,6\}$,
$COL(S)\ne COL(T)$. 
By Lemma~\ref{le:cap} there exists a rainbow set of size 
$|X|^{1/5}\ge k_2$.
\end{enumerate}

\end{proof}

\begin{lemma}\label{le:hramsey}
Let $a\ge 3$,  $c\ge 2$, and $k_1,\ldots,k_c\ge 1$.
Let $P=k_1\cdots k_{c-1}$ and $S=k_1+\cdots+k_{c-1}$.
\begin{enumerate}
\item
$R_a(k_1,k_2,\ldots,k_c) \le c^{R_{a-1}(k_1-1,k_2-1,\ldots,k_c-1)^{a-1}}$.
\item
Let 
$$
Z=\{\sigma \in [c]^* \st \hbox{ for all $i\in [c]$, $\sigma$ contains at most }k_i-1\hbox{ $i$'s  } \}.
$$
Then

$$\sum_{\sigma\in Z} |\sigma| \le P(k_c+S)^{S+2}.$$

\item
$R_3(k_1,\ldots,k_c) \le  c^{P(k_c+S)^{S+2}}.$
\item 
$R_4(k_1,\ldots,k_c) \le c^{c^{3P(k_c+S-c)^{S+2-c}}}.$
\item
For almost all $k$, $WER_3(e,k) \le \werthree.$
\item
$WER_3(6,k)\le \werthreesix.$
\item
$WER_3(13,k)\le \werthreeth.$
\end{enumerate}
\end{lemma}

\begin{proof}

\noindent 
1) \Erdosns-Rado~\cite{ErdosRado,sandow,GRS} showed that $\rakt \le  2^{\binom{\ramkmt+1}{a-1}}+a-2.$
This can be modified to show
$$R_a(k_1,k_2,\ldots,k_c) \le c^{\binom{R_{a-1}(k_1-1,\ldots,k_c-1)}{a-1}+a-2}.$$
Our result easily follows.

\medskip

\noindent
2) 
Clearly 

\[
\begin{array}{rl}
\sum_{\sigma\in Z} |\sigma|& = \sum_{j_1 = 0}^{k_1-1} \cdots \sum_{j_c = 0}^{k_c-1} (j_1 + \ldots + j_c) \frac{(j_1 + \ldots + j_c)!}{j_1! \cdots j_c!}\le \sum_{j_1 = 0}^{k_1-1} \cdots \sum_{j_c = 0}^{k_c-1} (k_c+S) \frac{(k_c+S)!}{k_c!}\cr
&\le P\sum_{j_c=0}^{k_c-1} (k_c+S)^{S+1}\le Pk_c(k_c+S)^{S+1}\le P(k_c+S)^{S+2}\cr
\end{array}
\]

\smallskip

\noindent
3) Conlon, Fox, and Sudakov~\cite{conlonfoxsud} have the best known
upper bounds on $R_3(k,k)$. Gasarch, Parrish, Sandow~\cite{sandow} have done a straightforward analysis
of their proof to extend it to $c$ colors. A modification of that proof yields
$$R_3(k_1,\ldots,k_c) \le c^{\sum_{\sigma\in Z} |\sigma|}.$$
Our result follows.

\smallskip

\noindent
4) This follows from parts 1 and 3. We could obtain a better result by
replacing $P$ by $(k_1-1)\cdots(k_c-1)$ but that would not help us later.

\smallskip

\noindent
5) By Lemma~\ref{le:wer3}

$$WER_3(e,k) \le R_4(e,e+2,e+2,e,e+2,e,k^5).$$

Let  $s(e)= e+(e+2)+(e+2)e+(e+2)+e=\se$.
Let $p(e)$ be the product of these terms.
By parts 2 and 5, for $k$ large, we have the following.

$$WER_3(e,k)
\le 7^{7^{3p(e)(k^5+s(e)-7)^{s(e)-5}}} 
\le 7^{7^{3p(e)(2k^5)^{s(e)-5}}} 
\le 7^{7^{(6p(e)k^5)^{s(e)-5}}} 
\le 2^{2^{(36p(e)k^5)^{s(e)-5}}} 
$$

Let $f(e)=(36p(e))^{s(e)-5}$. Then we have

$$WER_3(e,k)\le  2^{2^{f(e)k^{5s(e)-25}}} \le  2^{2^{k^{5s(e)-24}}}\le \werthree.$$

\noindent
6) This follows from part 5.

\noindent
7) This follows from part 5.

\end{proof}

\subsection{Geometric Lemmas}\label{se:geomtri}

\begin{definition}
Let $d\in\nat$.
If $p,q,r\in \real^d$ then let $AREA(p,q,r$) be the area of the triangle
with vertices $p,q,r$.
\end{definition}

The next lemma is 
Lemma 4 of~\cite{extri} whose proof is in the appendix of that paper.
They credit~\cite{regvert}, which is unavailable, with the proof.

\begin{lemma}\label{le:cyl}
Let $C_1,C_2,C_3$ be three cylinders with no pair of parallel axis.
Then $C_1\cap C_2 \cap C_3$ consists of at most 8 points.
\end{lemma}

\begin{lemma}\label{le:nohomogtri}~
\begin{enumerate}
\item
Let $p_1,\ldots,p_{n}$ be points in $\real^2$, no three collinear.
Color $\binom{[n]}{3}$ via
$COL(i,j,k)=AREA(p_i,p_j,p_k)$.
This coloring has no whomog set of size 6.
\item
Let $p_1,\ldots,p_{n}$ be points in $\real^3$, no three collinear.
Color $\binom{[n]}{3}$ via
$COL(i,j,k)=AREA(p_i,p_j,p_k)$.
This coloring has no whomog set of size 13.
\end{enumerate}
\end{lemma}

\begin{proof}

\noindent
1) Assume, by way of contradiction, there exists an $I$-whomog set of size 6.
By renumbering we can assume the $I$-whomog set is $\{1,2,3,4,5,6\}$.

\medskip

\noindent
{\bf Case 1:} $I=\{1\}$, $\{1,2\}$, or $\{2\}$.

We have $AREA(p_1,p_2,p_4)=AREA(p_1,p_2,p_5)$.
Thus $p_4$ and $p_5$ are either on a line parallel to $p_1p_2$
or are on different sides of $p_1p_2$. In the later case the
midpoint of $p_4p_5$ is on $p_1p_2$.

We have $AREA(p_1,p_3,p_4)=AREA(p_1,p_3,p_5)$.
Thus $p_4$ and $p_5$ are either on a line parallel to $p_1p_3$
or are on different sides of $p_1p_3$. In the later case the
midpoint of $p_4p_5$ is on $p_1p_3$.

We have $AREA(p_2,p_3,p_4)=AREA(p_2,p_3,p_5)$.
Thus $p_4$ and $p_5$ are either on a line parallel to $p_2p_3$
or are on different sides of $p_2p_3$. In the later case the
midpoint of $p_4p_5$ is on $p_2p_3$.

One of the following must happen.

\begin{itemize}
\item
Two of these cases have $p_4,p_5$ on the same side of the line.
We can assume that $p_4,p_5$ are on a line parallel to both $p_1p_2$ and $p_1p_3$.
Since $p_1,p_2,p_3$ are not collinear there is no such line.
\item
Two of these cases have $p_4,p_5$ on opposite sides of the line.
We can assume that the midpoint of $p_4p_5$ is on both $p_1p_2$
and $p_1p_3$. 
Since $p_1,p_2,p_3$ are not collinear the only point on both $p_1p_2$ and $p_1p_3$
is $p_1$. So the midpoint of $p_4,p_5$ is $p_1$. Thus $p_4,p_1,p_5$ are collinear
which is a contradiction.
\end{itemize}

Note that for $I=\{1\}$, $\{1,2\}$, or $\{2\}$ we used the line-point pairs
$$\{p_1p_2,p_1p_3,p_2p_3\}\times \{p_4,p_5\}.$$
For the rest of the cases we just specify which line-point pairs to use.

\medskip

\noindent
{\bf Case 2:} $I=\{3\}$ or $\{2,3\}$.
Use
$$\{p_4p_5,p_3p_5,p_3p_4\}\times \{p_1,p_2\}.$$

\medskip

\noindent
{\bf Case 3:} $I=\{1,3\}$
Use
$$\{p_1p_4,p_1p_5,p_1p_6\}\times\{p_2,p_3\}.$$
This is the only case that needs 6 points.

\medskip

\noindent
2) Assume, by way of contradiction, that there exists an $I$-whomog set of size 13.
By renumbering we can assume the $I$-whomog set is $\{1,\ldots,13\}$.

\medskip

\noindent
{\bf Case 1:} $I=\{1\}$, $\{1,2\}$, or $\{2\}$.

We have $AREA(p_1,p_2,p_4)=AREA(p_1,p_2,p_5)=\cdots=AREA(p_1,p_2,p_{12})$.
Hence $p_4,\ldots,p_{12}$ are all on a cylinder with axis $p_1p_2$.

We have $AREA(p_1,p_3,p_4)=AREA(p_1,p_3,p_5)=\cdots=AREA(p_1,p_3,p_{12})$.
Hence $p_4,\ldots,p_{12}$ are all on a cylinder with axis $p_1p_3$.

We have $AREA(p_2,p_3,p_4)=AREA(p_2,p_3,p_5)=\cdots=AREA(p_2,p_3,p_{12})$.
Hence $p_4,\ldots,p_{12}$ are all on a cylinder with axis $p_2p_3$.

Since $p_1,p_2,p_3$ are not collinear the three cylinders mentioned above
satisfy the premise of Lemma~\ref{le:cyl}. By that lemma there
are at most 8 points in the intersection of the three cylinders.
However, we just showed there are 9 such points. Contradiction.

Note that for $I=\{1\}$, $\{1,2\}$, or $\{2\}$  we used the line-point pairs
$$\{p_1p_2,p_1p_3,p_2p_3\}\times \{p_4,\ldots,p_{12}\}.$$
For the rest of the cases we just specify which line-point pairs to use.

\medskip

\noindent
{\bf Case 2:} $I=\{3\}$ or $\{2,3\}$.
Use
$$\{p_{11}p_{12},p_{10}p_{12},p_{10}p_{11}\}\times \{p_1,\ldots,p_9\}.$$

\medskip

\noindent
{\bf Case 3:} $I=\{1,3\}$
Use
$$\{p_1p_{11},p_1p_{12},p_1p_{13}\}\times\{p_2,\ldots,p_{10}\}.$$
This is the only case that needs 13 points.

\end{proof}

\subsection{Lower Bounds on $h_{3,2}(n)$ and $h_{3,3}(n)$}

\begin{theorem}~
\begin{enumerate}
\item
$h_{3,2}(n) \ge \restritwo$.
\item
$h_{3,3}(n) \ge  \restrithree$.
\end{enumerate}
\end{theorem}

\begin{proof}

\noindent
a) Let $P=\{p_1,\ldots,p_n\}$ be $n$ points in $\real^2$.
Let $COL$ be the coloring of $\binom{[n]}{3}$
defined by  $COL(i,j,k)=AREA(p_i,p_j,p_k)$.

Let $k$ be the largest integer such that
$$n\ge WER_3(6,k).$$
By Lemma~\ref{le:hramsey}.6 it will suffice to take $k=\restritwo$.
By the definition of $WER_3(6,k)$
there is either a whomog set of size $6$ or a rainbow
set of size $k$. By Lemma~\ref{le:nohomogtri}.a there cannot be such a whomog set,
hence must be a rainbow set of size $k$.

\medskip

\noindent
b) Let $P=\{p_1,\ldots,p_n\}$ be $n$ points in $\real^3$.
Let $COL$ be the coloring of $\binom{[n]}{3}$
defined by  $COL(i,j,k)=AREA(p_i,p_j,p_k)$.

Let $k$ be the largest integer such that
$$n\ge WER_3(13,k).$$
By Lemma~\ref{le:hramsey}.7 it will suffice to take $k=\restrithree$.
By the definition of $WER_3(13,k)$
there is either a whomog set of size $13$ or a rainbow
set of size $k$. By Lemma~\ref{le:nohomogtri}.b there cannot be such a whomog set,
hence must be a rainbow set of size $k$.
\end{proof}

Shelah~\cite{shelahcan} has shown that $ER_3(k)\le 2^{2^{p(k)}}$ 
for some polynomial $p(k)$.
We suspect that this result can be modified to obtain $ER_3(e,k)\le 2^{k^{f(e)}}$ for
some function $f$. 
If this is the case then there exists constants $c_2,c_3$ such that
$h_{3,2}(n)\ge \Omega((\log n)^{c_2})$
and
$h_{3,3}(n)\ge \Omega((\log n)^{c_3})$.
Furthermore, we believe that better upper bounds can be obtained for $WER_3(e,k)$
which will lead to larger values of $c_2$ and $c_3$.

For $d\ge 4$ we need the modification of Shelah's result and also some geometric
lemmas. We believe both are true, though the geometric lemmas look difficult.
Hence we believe that, for all $d$, there exists $c_d$ such that
$h_{3,d}(n)\ge \Omega((\log n)^{c_d})$.

\section{$h_{a,d}(\alpha)$ for Cardinals $\alpha$}\label{se:inf}

\begin{theorem}\label{th:aleph}~
\begin{enumerate}
\item
For all $d\ge 1$, $h_{2,d}(\aleph_0)=\aleph_0$.
\item
For all $d\ge 2$, $h_{3,2}(\aleph_0)=\aleph_0$.
\item
For all $d\ge 2$, $h_{3,3}(\aleph_0)=\aleph_0$.
\end{enumerate}
\end{theorem}

\begin{proof}

\noindent
1) Let $P$ be a countable subset of $R^d$. Define $COL:\binom{P}{2}$ via
$COL(x,y)=|x-y|$. By the (standard) infinite canonical Ramsey theorem 
there is either homog,  min-homog, max-homog, or rainbow set of size $\aleph_0$.
By Lemma~\ref{le:nohomog} the set cannot be homog, min-homog or max-homog.
Hence it is rainbow.

\noindent
2) This is similar to the proof of part 1, but using the (standard) 3-ary canonical Ramsey theorem and
Lemma~\ref{le:nohomogtri}.1.

\noindent
3) This is similar to the proof of part 1, but using the (standard) 3-ary canonical Ramsey theorem and
Lemma~\ref{le:nohomogtri}.2.
\end{proof}

\section{Speculation about Higher Dimensions}\label{se:spec}

\begin{definition}
Let $\Gamma_0(m)=m$ and,
for $a\ge 0$, $\Gamma_{a+1}(m)=2^{\Gamma_a(m)}$.
\end{definition}

To get lower bounds on $h_{a,d}(n)$ using our approach you need the following:
\begin{enumerate}
\item
Upper bounds on $ER_a(e,k)$.
\begin{enumerate}
\item
Lefmann and \Rodlns~\cite{BetterCanRamsey} have an upper bound on 
$ER_a(k)$ that involves $\Gamma_a$.
We are quite confident that this can be modified to obtain an upper bound on 
$ER_a(e,k)$ (if $e\ll k$ which is our case)
that involves $\Gamma_{a-1}$.
We are also quite confident that this can be modified to obtain an upper bound on 
$WER_a(e,k)$ 
that still involves $\Gamma_{a-1}$ but is better in terms of constants.
\item
Shelah~\cite{shelahcan} has an upper bound on
$ER_a(k)$ that involves $\Gamma_{a-1}$.
We suspect that this can be modified to obtain an upper bound on 
$ER_a(e,k)$ (if $e\ll k$ which is our case)
that involves $\Gamma_{a-2}$.
We also suspect that this can be modified to obtain an upper bound on 
$WER_a(e,k)$ that still involves $\Gamma_{a-2}$ but is better in terms of constants.
\end{enumerate}
\item
The following geometric lemma:
There exists a function $f(a,d)$ such that the following is true:
Let $p_1,\ldots,p_{n}$ be points in $\real^d$, no $a$ points in the same 
$(a-2)$-dimensional space.
Color $\binom{[n]}{a}$ via
$COL(i_1,\ldots,i_a)=VOLUME(p_{i_1},\ldots,p_{i_a})$.
If $I\subset [a]$ then 
this coloring has no $I$-whomog set of size $f(a,d)$.
We conjecture that this is true.
\end{enumerate}

If our suspicions about $ER_a(e,k)$ and our conjecture about geometry
are correct then the following is true:
For all $a,d$ there is a constant $\ep_{a,d}$ such that

$$(\forall a\ge 3)[h_{a,d}(n) = \Omega((\log^{(a-2)} n)^{\ep_{a,d}})].$$

\section{Open Questions}\label{se:open}

\begin{enumerate}
\item
Improve both the upper and lower bounds for $h_{a,d}$.
A combination of our combinatorial techniques and the geometric techniques
of the papers referenced in the introduction may lead to better lower bounds.
\item
We obtain $h_{2,1}(n)=\resdone$.
The known result, $h_{2,1}(n) = \Theta(n^{1/2})$, has a rather difficult
proof. It would be of interest to obtain an easier proof of either
the known result or a weaker version of it that is stronger
than what we have. An easy probabilistic argument yields $h_{2,1}(n)=\Omega(n^{1/4})$.
\item
Obtain upper bound on $ER_a(e,k)$, and 
geometric lemmas, in order to get nontrivial lower bounds on (1) $h_{3,d}$ for $d\ge 4$, and
(2) $h_{a,d}$ for $d\ge 4$, and $a\ge d$. 
See Section~\ref{se:spec} for more thoughts on this.
\item
Look at a variants of $h_{a,d}(n)$ with different metrics on $\real^d$ or in other metric spaces entirely.
\item
Look at a variant of $h_{a,d}(n)$, which we call $h_{a,d}'(n)$,
where the only condition on the points is that they are not all on the same
$(a-2)$-dimensional space.
Using the $n^{1/d} \times \cdots \times n^{1/d}$ grid it is easy to show
that, for $a,d\ll n$,  $h_{a,d}'(n) \le O(n^{(a-1)/a})$.
\item
We showed 
$h_{3,2}(\aleph_0)=\aleph_0$ and
$h_{3,3}(\aleph_0)=\aleph_0$. 
We conjecture that, for $\aleph_0 \le \alpha \le 2^{\aleph_0}$, 
$h_{a,d}(\alpha)=\alpha$.
This may require a canonical Ramsey theorem where the graph has $\alpha$ vertices
and the coloring function is well behaved.
\end{enumerate}

\section{Acknowledgments}

We would like to thank Tucker Bane, Andrew Lohr, Jared Marx-Kuo, and Jessica Shi for helpful discussions.
We would like to thank David Conlon and Jacob Fox for thoughtful discussions,
many references and observations, encouragement, and advice on this paper.

%\bibliographystyle{abbrv}
%\bibliography{bibfile}

\end{document}